\documentclass[11pt]{article}
\usepackage{graphicx}
\usepackage{amsmath}
\usepackage{amsfonts}
\usepackage{amssymb}

\usepackage{cite}
\providecommand{\U}[1]{\protect\rule{.1in}{.1in}}
\newtheorem{theorem}{Theorem}[section]

\newtheorem{corollary}[theorem]{Corollary}

\newtheorem{definition}[theorem]{Definition}

\newtheorem{proposition}[theorem]{Proposition}
\newtheorem{remark}[theorem]{Remark}

\newenvironment{proof}[1][Proof]{\textbf{#1.} }{\hfill\rule{0.5em}{0.5em}}
{\catcode`\@=11\global\let\AddToReset=\@addtoreset
\AddToReset{equation}{section}

\AddToReset{theorem}{section}

\begin{document}
\title{Wiener criteria for existence of large solutions of  \\quasilinear elliptic equations with absorption}
\author{
 {\bf Nguyen Quoc Hung\thanks{ E-mail address: Hung.Nguyen-Quoc@lmpt.univ-tours.fr}}\\[1mm]
 {\bf Laurent V\'eron\thanks{ E-mail address: Laurent.Veron@lmpt.univ-tours.fr}}\\[2mm]
{\small Laboratoire de Math\'ematiques et Physique Th\'eorique, }\\
{\small  Universit\'e Fran\c{c}ois Rabelais,  Tours,  FRANCE}}


\newcommand{\abs}[1]{\left |#1\right |}
\newcommand{\myfrac}[2]{{\displaystyle \frac{#1}{#2} }}
\newcommand{\myint}[2]{{\displaystyle \int_{#1}^{#2}}}



\date{}
\maketitle\medskip

\date{}  
\abstract{ We obtain sufficient  conditions, expressed in terms of Wiener type tests involving Hausdorff or Bessel capacities, for the existence of large solutions to equations (1) $-\Delta_pu+e^{ u}-1=0$ or (2) $-\Delta_pu+ u^q=0$ in a bounded domain $\Omega$ 
when $q>p-1>0$. We apply our results to equations (3) $-\Delta_pu+a\abs{\nabla u}^{q}+bu^{s}=0$,
 (4) $\Delta_p u+u^{-\gamma}=0$ with $1<p\le 2$, $1\le q\le p$, $a>0, b>0$ and $q>p-1$, $s\geq p-1$,  $\gamma>0$.\medskip
 
 \smallskip


\section{Introduction}

Let $\Omega$ be a bounded domain in $\mathbb{R}^N$ ($N\geq 2$) and $1<p\leq N$. We denote $\Delta_pu=\text{div}(\abs{\nabla u}^{p-2}\nabla u)$, $\rho(x)=\text{dist}(x,\partial\Omega)$. In this paper we study some questions relative to the existence of solutions to the problem
\begin{equation}\label{3hvIn1}\begin{array} {ll}
-\Delta_pu+g(u)=0\qquad\text{in }\Omega\\
\phantom{-}
\displaystyle\lim_{\rho(x)\to 0}u(x)=\infty
\end{array}\end{equation}
where $g$ is a continuous nondecreasing function vanishing at $0$, and most often $g(u)$ is either $\text{sign}(u)(e^{ \abs u} -1)$ or $\abs u^{q-1}u$ with $q>p-1$. A solution to problem \eqref{3hvIn1} is called a {\it large solution.} When the domain is regular in the sense that the Dirichlet problem with continuous boundary data $\phi$
\begin{equation}\label{3hvIn2}\begin{array} {ll}
-\Delta_pu+g(u)=0\qquad\text{in }\Omega,\\
u-\phi \in W^{1,p}_{0}(\Omega), u\in W_{\text{loc}}^{1,p}(\Omega)\cap L^\infty(\Omega),
\end{array}\end{equation}
admits a solution $u\in C(\overline{\Omega})$, it is clear that problem (\ref{3hvIn1}) admits a solution provided problem $-\Delta_pu+g(u)=0~~\text{in }\Omega$ having a maximal solution, see \cite[Chapter 5]{33MV}. It is known that a necessary and sufficient condition for the solvability of problem (\ref{3hvIn2}) in case $g(u)\equiv0$ is the {\it Wiener criterion}, due to Wiener \cite{33WI} when $p=2$ and
Maz'ya \cite{33MA}, Kilpelainen and Mal\'y \cite{33KM} when $p\neq 2$, in general case is proved by Mal\'y and Ziemer  \cite{33MZ}. This condition is
\begin{equation}\label{3hvIn3}\begin{array} {ll}
\myint{0}{1}\left(\myfrac{C_{1,p}(B_t(x)\cap\Omega^c)}{t^{N-p}}\right)^{\frac{1}{p-1}}\myfrac{dt}{t}=\infty\qquad\forall x\in\partial\Omega,
\end{array}\end{equation}
where $C_{1,p}$ denotes the capacity associated to the space $W^{1,p}(\mathbb{R}^N)$.  The existence of a maximal solution is guaranteed for a large class of nondecreasing nonlinearities $g$ satisfying the Vazquez condition\cite{33VA} 
\begin{equation}\label{3hvIn3b}\begin{array} {ll}
\myint{a}{\infty}\myfrac{dt}{\sqrt[\,p]{G(t)}}<\infty\quad\text{where }\; G(t) =\myint{0}{t}g(s)ds
\end{array}\end{equation}
for some $a>0$. This is an extension of the Keller-Osserman condition \cite{33KE}, \cite{33OS}, which is the above relation when $p=2$. If for $R>diam(\Omega)$ there exists a function $v$ which satisfies
\begin{equation}\label{3hvIn4}\begin{array} {ll}
-\Delta_pv+g(v)=0\qquad&\text{in } B_R\setminus\{0\},\\
\phantom{-\Delta_p+g(v)}
v=0&\text{on } \partial B_R,\\
\phantom{....-}
\displaystyle\lim_{x\to 0}v(x)=\infty,
\end{array}\end{equation}
then it is easy to see that the maximal solution $u$ of 
\begin{equation}\label{3hvIn5}\begin{array} {ll}
-\Delta_pu+g(u)=0\qquad\text{in } \Omega
\end{array}\end{equation}
is a large  solution, without any assumption on the regularity of $\partial\Omega$. Indeed,  $ x\mapsto v(x-y)$ is a solution of \eqref{3hvIn5} in $\Omega$ for all $y\in \partial\Omega$, thus $u(x)\geq v(x-y)$ for any $x\in\Omega, y\in \partial\Omega$. It follows $\lim_{\rho(x)\to 0}u(x)=\infty$ since $\lim_{z\to 0}v(z)=\infty $.\medskip \\   Remark that the existence of a (radial) solution to problem (\ref{3hvIn4}) needs the fact that equation (\ref{3hvIn5}) admits solutions with isolated singularities, which is usually not true if the growth of $g$ is too strong since Vazquez and V\'eron prove in \cite{33VV} that if 
\begin{equation}\label{3hvIn5+}\begin{array} {ll}
\displaystyle\liminf_{\abs{r}\to\infty}\abs{r}^{-\frac{N(p-1)}{N-p}}sign (r)g(r)>0~~\text{ with } ~p<N,
\end{array}\end{equation}
isolated singularities of solutions of (\ref{3hvIn5}) are removable. Conversely, if $p-1<q<\frac{N(p-1)}{N-p}$ with $p<N$,  Friedman and V\'eron \cite{33FV} characterize the behavior of positive singular solutions to
\begin{equation}\label{3hvIn5++}\begin{array} {ll}
-\Delta_pu+u^q=0
\end{array}\end{equation}
with an isolated singularities.
 In 2003, Labutin \cite{33LA} show that a necessary and sufficient condition in order the following problem be solvable
\begin{equation*}\begin{array} {ll}
-\Delta u+\abs u^{q-1}u=0\qquad\text{in }\Omega,\\
\phantom{,,,-}
\displaystyle\!\!\lim_{\rho(x)\to 0}u(x)=\infty,
\end{array}\end{equation*}
 is that 
\begin{equation*}\begin{array} {ll}
\myint{0}{1}\myfrac{C_{2,q'}(B_t(x)\cap \Omega^c)}{t^{N-2}}\myfrac{dt}{t}=\infty\qquad\forall x\in\partial\Omega,
\end{array}\end{equation*}
where $C_{2,q'}$ is the capacity associated to the Sobolev space $W^{2,q'}(\mathbb{R}^N)$ and $q'=q/(q-1)$, $N\geq 3$. Notice that this condition is always satisfied if $q$ is subcritical, i.e. $q<N/(N-2)$. We refer to \cite{33MV} for other related results. Concerning the exponential case of problem (\ref{3hvIn1}) nothing is known, even in the case $p=2$, besides the simple cases already mentioned. \smallskip

In this article we give sufficient conditions, expressed in terms of Wiener tests, in order problem (\ref{3hvIn1}) be solvable  in the two cases $g(u)=\text{sign}(u)(e^{|u|}-1)$ and $g(u)= \abs u^{q-1}u$, $q>p-1$.  For $1<p\leq N$, we denote by $\mathcal{H}_1^{N-p}(E)$ the Hausdorff capacity of a set $E$
defined by
\begin{equation*}\begin{array} {ll}
\mathcal{H}_1^{N-p}(E)=\inf\left\{\displaystyle\sum_jh^{N-p}(B_j): E\subset\bigcup B_j,\, diam (B_j)\leq 1\right\}
\end{array}\end{equation*}
where the $B_j$ are balls and $h^{N-p}(B_r)=r^{N-p}$.
Our main result concerning the exponential case is the following\medskip

\noindent{\bf Theorem 1}. {\it Let $N\geq 2$ and $1<p\leq N$.
If \begin{equation}\label{3hvco1}
\int_{0}^{1}\left(\frac{\mathcal{H}_1^{N-p}(\Omega^c\cap B_r(x))}{r^{N-p}}\right)^{\frac{1}{p-1}}\frac{dr}{r}=+\infty~~\forall x\in \partial \Omega,
\end{equation}
 then there exists  $u\in C^{1}(\Omega)$ satisfying 
\begin{equation}\label{3hvIn9}\begin{array} {lll}
   - \Delta_p u + e^{u}-1 = 0\qquad\text{in }\Omega,  \\ 
   \phantom{e^{\lambda u}}
\displaystyle\,\!\lim_{\rho(x)\to 0}u(x)=\infty.
\end{array}
\end{equation}
} 
\medskip

Clearly, when $p=N$, we have $\mathcal{H}_1^{N-p}(\{x_0\})=1$ for all $x_0\in\mathbb{R}^N$ thus, \eqref{3hvco1} is true for any open domain $\Omega$. 

We also obtain a sufficient condition for the existence of a large solution in the power case expressed in terms of some $C_{\alpha,s}$ Bessel capacity in $\mathbb{R}^N$  associated to the Besov space $B^{\alpha,s}(\mathbb{R}^N)$.\medskip

\noindent{\bf Theorem 2}. {\it Let $N\geq 2$, $1<p<N$ and $q_1>\frac{N(p-1)}{N-p}$. If 
\begin{equation}\label{3hvco2}
\int\limits_0^1 {{{\left( {\frac{{C_{p,\frac{q_1}{{q_1 - p + 1}}}\left( {{\Omega ^c} \cap B_r(x)} \right)}}{{{r^{N - p}}}}} \right)}^{\frac{1}{{p - 1}}}}\frac{{dr}}{r}}  =  + \infty~~~\forall x\in \partial \Omega,
\end{equation}
then, for any $p-1<q<\frac{pq_1}{N}$ there exists  $u\in C^{1}(\Omega)$ satisfying 
\begin{equation}\label{3hvIn8}\begin{array} {lll}
   - \Delta_p u + u^q = 0\qquad\text{in }\Omega,  \\ 
   \phantom{}
\displaystyle\lim_{\rho(x)\to 0}u(x)=\infty.
\end{array}
\end{equation}}

We can see that condition \eqref{3hvco1} implies \eqref{3hvco2}. 
In view of Labutin's theorem this previous result is not optimal in the case $p=2$, since the involved capacity is $C_{2,q'_1}$ with $q'_1$ and thus there exists a solution to
\begin{equation*}\begin{array} {lll}
   - \Delta_p u + u^{q_1}  = 0\qquad\text{in }\Omega  \\ 
   \phantom{_1}
\displaystyle\lim_{\rho(x)\to 0}u(x)=\infty
\end{array}
\end{equation*}
with $q_1>q$. \smallskip

At end we apply the previous theorem to quasilinear viscous Hamilton-Jacobi equations: 
\begin{equation}\label{3hvIn10a}\begin{array} {lll}
   - \Delta_p u +a\abs{\nabla u}^{q}+ b|u|^{s-1}u = 0\qquad\text{in }\Omega,  \\ 
   \phantom{}
\displaystyle\!~~~~~u\in C^1(\Omega), \lim_{\rho(x)\to 0}u(x)=\infty.
\end{array}
\end{equation}
For $q_1>p-1$ and $1<p\leq 2$, if equation \eqref{3hvIn8} admits a solution with $q=q_1$, then for any $a>0, b> 0$ and $q\in(p-1,\frac{pq_1}{q_1+1})$,  $ s\in [p-1,q_1)$
there exists a positive solution to
\eqref{3hvIn10a}. Conversely, if for some $a,b>0$, $s>p-1$ there exists a solution to equation \eqref{3hvIn10a} with $1<q=p\leq 2$, 
then for any $q_1>p-1$, $1\le q_1\le p$, $s_1\geq p-1$, $a_1,b_1>0$ there exists a positive solution to equation \eqref{3hvIn10a} with parameters $q_1,s_1,a_1,b_1$ replacing $q,s,a,b$. Moreover, we also prove that the previous statement holds if for some $\gamma>0$ there exists $u\in C(\overline\Omega)\cap C^1(\Omega)$, $u>0$ in $\Omega$ satisfying
\begin{equation*}\begin{array} {lll}
   - \Delta_p u +u^{-\gamma} = 0~~\text{in }\Omega,  \\ 
   \phantom{- \Delta u  +u^{-\gamma}}
u=0~~\text{on } \partial\Omega.
\end{array}
\end{equation*}
We would like to remark that the case $p=2$ was studied in \cite{33LALI}. In particular, if the boundary of $\Omega$ is smooth  then \eqref{3hvIn10a} has a solution with $s=1$ and $1<q\leq 2,a>0,b>0$. 

\section{Morrey classes and Wolff potential estimates}
\setcounter{equation}{0}
In this section we assume that $\Omega$ is a bounded open subset of $\mathbb{R}^N$ and $1<p<N$. We also denote by $B_r(x)$ the open ball of center $x$ and radius $r$ and $B_r=B_r(0)$. We also recall that a solution of \eqref{3hvIn1} belongs to $C_{loc}^{1,\alpha}(\Omega)$ for some $\alpha\in (0,1)$, and is more regular (depending on $g$) on the set $\{x\in\Omega:\abs{\nabla u(x)}\neq 0\}$.

\begin{definition}\label{3hvMorrey} A function $f\in L^1(\Omega)$  belongs to the Morrey space $\mathcal{M}^s(\Omega)$, $1\leq s\leq\infty$, if there is a constant $K$ such that 
\begin{equation*}
\int_{\Omega\cap B_r(x)}|f|dy\le K r^{\frac{N}{s'}}~~\forall r>0,\,\forall x\in\mathbb{R}^N.
\end{equation*}
The norm is defined as the smallest constant $K$ that satisfies this inequality; it is denoted by $||f||_{\mathcal{M}^s(\Omega)}$. Clearly $L^s(\Omega)\subset \mathcal{M}^{s}(\Omega)$.
\end{definition}

\begin{definition}\label{3hvWolff} Let $R\in (0,\infty]$ and $\mu\in\mathfrak{M}^b_+(\Omega)$, the set of nonnegative and bounded Radon measures in $\Omega$. We define the ($R$-truncated) Wolff potential of $\mu$ by
\begin{equation*}
{\bf W}^R_{1, p}[\mu](x)=\myint{0}{R}\left(\myfrac{\mu(B_t(x))}{t^{N- p}}\right)^{\frac{1}{p-1}}\myfrac{dt}{t}\qquad\forall x\in \mathbb{R}^N,
\end{equation*}
and the ($R$-truncated) fractional maximal potential  of $\mu$ by
\begin{equation*}
{\bf  M}_{p,R}[\mu](x)=\sup_{0<t<R}\frac{\mu(B_t(x))}{t^{N-p}} \qquad\forall x\in \mathbb{R}^N,
\end{equation*}
where the measure is extended by $0$ in $\Omega^c$.
\end{definition}
We recall a result proved in \cite{33HoJa} (see also \cite[Theorem 2.4]{33BVHV}).

\begin{theorem}  \label{3hvth2} Let $\mu$ be a nonnegative Radon measure in $\mathbb{R}^N$. There exist positive constants $C_1,C_2$ depending on $N,p$  such that 
$$\int_{2B}\exp(C_1\mathbf{W}^R_{1,p}[\chi_B\mu])dx\le C_2r^{N},$$ 
 for all $B=B_r(x_0)\subset \mathbb{R}^N$, $2B=B_{2r}(x_0)$, $R>0$ such that $||{\bf  M}_{p,R}[\mu]||_{L^\infty(\mathbb{R}^N)}\le 1$. 
\end{theorem}
For $k\geq 0$, we set $T_k(u)=sign (u)\min\{k,\abs u\}$.

\begin{definition}\label{3hvsupersol} Assume $f\in L^1_{loc}(\Omega)$. We say that a measurable function $u$ defined in $\Omega$ is a renormalized  supersolution of 
\begin{equation}\label{3hvPr4}-\Delta_pu+f=0\qquad\text{in }\Omega
\end{equation}
if,  for any $k>0$, $T_k(u)\in W^{1,p}_{loc}(\Omega)$, $\abs{\nabla u}^{p-1}\in L^1_{loc}(\Omega)$ and there holds
\begin{equation*}
\int_{\Omega}(|\nabla T_k(u)|^{p-2}\nabla T_k(u)\nabla\varphi+f\varphi)dx\geq  0
\end{equation*}
for all $\varphi\in W^{1,p}(\Omega)$ with compact support in $\Omega$ and such that $0\le\varphi \leq k-T_k(u)$, and if  $-\Delta_pu+f$ is a positive  distribution in $\Omega$. 
\end{definition}
The following result is proved in \cite[Theorem 4.35]{33MZ}.
\begin{theorem} If $f\in \mathcal{M}^{\frac{N}{p-\epsilon}}(\Omega)$ for some $\epsilon\in (0,p)$, $u$ is a nonnegative renormalized supersolution of \eqref{3hvPr4} and set $\mu:=-\Delta_pu+f$. Then  there holds 
\begin{equation*}
u(x)+||f||^{\frac{1}{p-1}}_{\mathcal{M}^{\frac{N}{p-\varepsilon}}(\Omega)}\geq C \mathbf{W}^{\frac{r}{4}}_{1,p}[\mu](x)
\qquad\forall x\in\Omega\,\text { s.t. }B_r(x)\subset \Omega,
\end{equation*}
for some $C$ depending only on $N,p,\varepsilon,diam(\Omega)$.
\end{theorem} 

Concerning renormalized solutions (see \cite{33DMOP} for the definition) of
 \begin{equation}\label{3hvPr7}
 -\Delta_pu+f=\mu\qquad\text{in }\Omega,
\end{equation}
where $f\in L^1(\Omega)$ and $\mu\in\mathfrak{M}_+^b(\Omega)$, we have
\begin{corollary}\label{3hvcoro1} Let $f\in \mathcal{M}^{\frac{N}{p-\epsilon}}(\Omega)$ and $\mu\in \frak{M}^b_+(\Omega)$. If $u$ is a renormalized solution to (\ref{3hvPr7})  and $\inf_{\Omega}u>-\infty$ then  there exists a positive constant  $C$ depending only on $N,p,\varepsilon,diam(\Omega)$ such that 
 \begin{equation*}u(x)+||f||^{\frac{1}{p-1}}_{\mathcal{M}^{\frac{N}{p-\varepsilon}}(\Omega)}\geq \inf_{\Omega}u+ C \mathbf{W}^{\frac{d(x,\partial\Omega)}{4}}_{1,p}[\mu](x)~~\forall x\in\Omega. 
 \end{equation*}
\end{corollary}
The next result, proved in \cite[Theorem 1.1, 1.2]{33BVHV}, is an important tool for the proof of Theorems \textbf{1} and \textbf{2}. Before presenting we introduce the notation.  
\begin{definition}Let $s>1$ and $\alpha>0$.
We denote by $C_{\alpha,s}(E)$  the Bessel capacity of Borel set $E\subset \mathbb{R}^N$,
\begin{equation*}
C_{\alpha,s}(E)=\inf\{||\phi||^s_{L^s(\mathbb{R}^N)}:\phi\in L^s_+(\mathbb{R}^N),~~G_\alpha*\phi \geq \chi_E\}
\end{equation*}
where $\chi_E$ is the characteristic function of $E$ and $G_\alpha$ the Bessel kernel of order $\alpha$. \\
We say that a measure $\mu$ in $\Omega$ is absolutely continuous with respect to the capacity $C_{\alpha,s}$ in $\Omega$ if
\begin{align*}
\text{ for all }~E\subset \Omega, E\text{ Borel}, C_{\alpha,s}(E)=0 \Rightarrow |\mu|(E)=0.
\end{align*}
\end{definition}
\begin{theorem}\label{3hv06054} Let $\mu\in \mathfrak{M}_+^b(\Omega)$ and $q>p-1$.
\begin{description}
\item[a.] If $\mu$ is absolutely continuous with respect to the capacity $C_{p,\frac{q}{q+1-p}}$ in $\Omega$, then there exists a nonnegative renormalized solution $u$ to equation 
\begin{equation*}
 \begin{array}{ll}
  - \Delta_p u + u^q =\mu ~~~&\text{in }\;\Omega, 
  \\ \phantom{- \Delta_p  + u^q}
 u = 0&\text{on }\;\partial \Omega,  \\ 
 \end{array} 
\end{equation*} which satisfies 
\begin{align}\label{3hvine}
u(x)\leq C \mathbf{W}_{1,p}^{2\text{diam}(\Omega)}[\mu](x)~~\forall x\in\Omega
\end{align}
where $C$ is a positive constant depending on $p$ and $N$. 
\item[b.] If $\exp(C\mathbf{W}_{1,p}^{2\text{diam}(\Omega)}[\mu])\in L^1(\Omega)$ where $C$ is the previous constant, then  there exists a nonnegative renormalized solution $u$ to equation 
\begin{equation*}
 \begin{array}{ll}
  - \Delta_p u + e^u-1 =\mu ~~~&\text{in }\;\Omega, 
  \\ \phantom{- \Delta_p  + e^u-1}
 u = 0&\text{on }\;\partial \Omega,  \\ 
 \end{array} 
\end{equation*} which satisfies \eqref{3hvine}.
\end{description}
\end{theorem}

\section{Estimates from below}

If $G$ is any domain in $\mathbb{R}^N$ with a compact boundary and $g$ is nondecreasing, $g(0)=g^{-1}(0)=0$ and satisfies \eqref{3hvIn5+}) there always exists a maximal solution to \eqref{3hvIn5} in $G$. It is constructed as the limit, when $n\to\infty$, of the solutions of 
\begin{equation}\label{3hvb-n}\begin{array}{lll}
-\Delta_pu_n+g(u_n)=0\qquad&\text{in }G_n\\
\phantom{-,}
\!\displaystyle\lim_{\rho_n(x)\to 0}u_n(x)=\infty\\
\phantom{--,}
\!\displaystyle\lim_{\abs x\to\infty}u_n(x)=0\quad&\text{if $G_n$ is unbounded},
\end{array}\end{equation}
where $\{G_n\}_n$ is a sequence of smooth domains such that  $G_n\subset \overline{G}_n\subset G_{n+1}$ for all $n$, $\{\partial G_n\}_n$ is a bounded and $\bigcup\limits_{n = 1}^\infty  {{G _n}}  = G$ and $\rho_n(x):=\text{dist}(x,\partial G_n)$. Our main estimates are the following.
\begin{theorem}\label{3hvth3}
Let $K\subset B_{1/4}\backslash\{0\}$ be a compact set and let $U_j\in C^{1}(K^c)$, $j=1,2$,
 be the maximal solutions of  
\begin{equation}\label{3hvb-1}
-\Delta_p u +e^{u}-1=0\qquad\text{in }K^c
\end{equation}
for $U_1$ and
\begin{equation}\label{3hvb-2}
 -\Delta_p u +u^q=0\qquad\text{in }K^c
\end{equation}
for $U_2$, where $p-1<q<\frac{pq_1}{N}$. Then there exist constants $C_k$, $k=1,2,3,4$, depending on $N$, $p$ and $q$ such that
\begin{equation}\label{3hvinea}
U_1(0)\geq -C_1+C_2\int_{0}^{1}\left(\frac{\mathcal{H}_1^{N-p}(K\cap B_r)}{r^{N-p}}\right)^{\frac{1}{p-1}}\frac{dr}{r},
\end{equation}
and
\begin{equation}\label{3hvineb}
U_2(0)\geq -C_3+C_4\int_0^1 {{{\left( {\frac{{C{_{{p},\frac{{{q_1}}}{{{q_1} - p + 1}}}}(K \cap B_r)}}{{{r^{N - p}}}}} \right)}^{\frac{1}{{p - 1}}}}} \frac{{dr}}{r}.
\end{equation}
\end{theorem}
\begin{proof}
{\it 1}. For $j\in \mathbb{Z}$ define $r_j=2^{-j}$ and $S_j=\{x:r_j\le |x|\le r_{j-1}\}$, $B_j=B_{r_j}$. Fix a positive integer $J$ such that $K\subset\{x:r_{J}\le |x|<1/8\}$. Consider the sets $K\cap S_j$ for $j=3,...,J$. By \cite[Theorem 3.4.27]{33TU}, there exists $\mu_j\in \mathfrak{M}^+(\mathbb{R}^N)$
such that $\text{supp}(\mu_j) \subset K\cap S_j$, ${\left\| {{{\bf M}_{p,1}}{[\mu _j]}} \right\|_{{L^\infty }(\mathbb{R}^N)}} \le 1$ and
  $$c_1^{-1}\mathcal{H}_1^{N-p}(K\cap S_j)\le \mu_j(\mathbb{R}^N)\le c_1\mathcal{H}_1^{N-p}(K\cap S_j)~~~~\forall j,$$ 
  for some $c_1=c_1(N,p)$.\\
 Now, we will show that for $\varepsilon=\varepsilon(N,p)>0$ small enough, there holds,
 \begin{equation}
\label{3hv06051}A:=\int_{B_1}\exp\left(\varepsilon{\bf W}^{1}_{1,p}\left[\sum_{k=3}^{J} \mu_{k}\right](x)\right) dx\leq c_2,
 \end{equation}
 where  $c_2$ does not depend on $J$.  \\
 Indeed, define $\mu_j\equiv 0$ for all $j\geq J+1$ and $j\le 2$. We have 
 \[A= \sum\limits_{j = 1}^\infty  {\int_{{S_j}} {\exp } \left( {{\varepsilon }{\bf W}_{1,p}^1\left[\sum\limits_{k = 3}^J {{\mu _k}} \right](x)} \right)dx}. \]
 Since for any $j$ \[{\bf W}_{1,p}^1\left[\sum\limits_{k = 3}^J {{\mu _k}} \right] \le c(p)  {{\bf W}_{1,p}^1\left[\sum\limits_{k \geq j + 2} {{\mu _k}} \right] + c(p){\bf W}_{1,p}^1\left[\sum\limits_{k \le j - 2} {{\mu _k}} \right]+ c(p)\sum\limits_{k = \max \{ j - 1,3\} }^{j + 1} {{\bf W}_{1,p}^1[{\mu _k}]} }, \] with $c(p)= \max \{ 1,{5^{\frac{{2 - p}}{{p - 1}}}}\}$  and $\exp (\sum\limits_{i = 1}^5 {{a_i}} ) \le \sum\limits_{i = 1}^5 {\exp (5{a_i})}$ for all $a_i.$
 Thus,
 \begin{align*}
   A &\leq \sum\limits_{j = 1}^\infty  {\myint{{S_j}}{} {\exp } \left( {c_3{\varepsilon }{\bf W}_{1,p}^1\left[\sum\limits_{k \geq j + 2} {{\mu _k}} \right](x)} \right)dx} + \sum\limits_{j = 1}^\infty  {\myint{{S_j}}{} {\exp } \left( {c_3{\varepsilon}{\bf W}_{1,p}^1\left[\sum\limits_{k \le j - 2} {{\mu _k}}\right ](x)} \right)dx}\\&~+\sum\limits_{j = 1}^\infty  {\sum\limits_{k = \max (j - 1,3)}^{j + 1} {\myint{{S_j}}{} {\exp } \left( {c_3{\varepsilon }{\bf W}_{1,p}^1[{\mu _k}](x)} \right)dx} } :=A_1+A_2+A_3, \text{ with } c_3=5c(p).
 \end{align*} 
 {\it Estimate of $A_3$}:
   We apply Theorem \ref{3hvth2} for $\mu= \mu_{k}$ and $B=B_{k-1}$,\[\int_{2{B_{k - 1}}} \exp\left( { c_3{\varepsilon }{\bf W}_{1,p}^1[{\mu _k}](x)} \right)dx \le c_4r_{k-1}^N\] 
   with $c_3{\varepsilon } \in (0,C_1]$,  the constant $C_1$ is in Theorem \ref{3hvth2}.
    In particular, 
    \[\int_{S_j} \exp\left( {c_3 {\varepsilon }{\bf W}_{1,p}^1[{\mu _k}](x)} \right)dx \le {c_4}r_{k-1}^N~\text{ for } k=j-1,j,j+1,\]   
    which implies
    \begin{equation}\label{3hvI3}
    A_3\le c_5\sum_{j=1}^{+\infty} r_j^N=c_5<\infty. 
    \end{equation}
  {\it  Estimate of $A_1$}:  Since $\sum\limits_{k \geq j + 2} {{\mu _k}} \left( {{B_t}(x)} \right) = 0$ for all $x \in {S_j},t\in(0,{r_{j + 1}})$.  Thus, 
  \begin{align*}
  {A_1} &= \sum\limits_{j = 1}^\infty  {\int_{{S_j}} {\exp } \left( {c_3{\varepsilon }\int\limits_{{r_{j + 1}}}^1 {{{\left( {\frac{{\sum\limits_{k \geq j + 2} {{\mu _k}} ({B_t}(x))}}{{{t^{N - p}}}}} \right)}^{\frac{1}{{p - 1}}}}\frac{{dt}}{t}} } \right)dx}\\&\le \sum\limits_{j = 1}^\infty  { {\exp } \left( {c_3{\varepsilon}\frac{{p - 1}}{{N - p}}{{\left( {\sum\limits_{k \geq j + 2} {{\mu _k}({S_k})} } \right)}^{\frac{1}{{p - 1}}}}r_{j + 1}^{ - \frac{{N - p}}{{p - 1}}}} \right)}|S_j|.
  \end{align*}
    Note that ${\mu _k}({S_k}) \le {\mu _k}({B_{{r_{k - 1}}}}(0)) \leq r_{k - 1}^{N - p}$, which leads to
    \[{\left( {\sum\limits_{k \geq j + 2} {{\mu _k}({S_k})} } \right)^{\frac{1}{{p - 1}}}}r_{j + 1}^{ - \frac{{N - p}}{{p - 1}}} \le {\left( {\sum\limits_{k \geq j + 2} {r_{k - 1}^{N - p}} } \right)^{\frac{1}{{p - 1}}}}r_{j + 1}^{ - \frac{{N - p}}{{p - 1}}} = {\left( {\sum\limits_{k \geq 0} {r_k^{N - p}} } \right)^{\frac{1}{{p - 1}}}} = {\left( {\frac{1}{{1 - {2^{ - (N - p)}}}}} \right)^{\frac{1}{{p - 1}}}}.\]
Therefore
\begin{equation}\label{3hvI1}
 {A_1} \le\exp \left( {c_3{\varepsilon }\frac{{p - 1}}{{N - p}}{{\left( {\frac{1}{{1 - {2^{ - (N - p)}}}}} \right)}^{\frac{1}{{p - 1}}}}} \right)|{B_1}| = {c_6}. 
\end{equation}
  {\it   Estimate of $A_2$}: for $x\in S_{j}$, 
      \[ \displaystyle{\bf W}_{1,p}^1\left[\sum\limits_{k \le j - 2} {{\mu _k}} \right](x) = \displaystyle \int\limits_{{r_{j - 1}}}^1 {{{\left( {\frac{{\sum\limits_{k \le j - 2} {{\mu _k}({B_t}(x))} }}{{{t^{N - p}}}}} \right)}^{\frac{1}{{p - 1}}}}\frac{{dt}}{t}}  = \sum\limits_{i = 1}^{j - 1} {\int\limits_{{r_i}}^{{r_{i - 1}}} {{{\left( {\frac{{\sum\limits_{k \le j - 2} {{\mu _k}({B_t}(x))} }}{{{t^{N - p}}}}} \right)}^{\frac{1}{{p - 1}}}}\frac{{dt}}{t}} }. \]
      Since ${r_i} < t < {r_{i - 1}}$, $\sum\limits_{k \le i - 2} {{\mu _k}({B_t}(x))}  = 0,\forall i = 1,...,j - 1$, thus
      \begin{align*}
      {\bf W}_{1,p}^1\left[\sum\limits_{k \le j - 2} {{\mu _k}} \right](x) &= \sum\limits_{i = 1}^{j - 1} {\int\limits_{{r_i}}^{{r_{i - 1}}} {{{\left( {\frac{{\sum\limits_{k=i-1}^{j-2} {{\mu _k}({B_t}(x))} }}{{{t^{N - p}}}}} \right)}^{\frac{1}{{p - 1}}}}\frac{{dt}}{t}} } \le \sum\limits_{i = 1}^{j - 1} {\int\limits_{{r_i}}^{{r_{i - 1}}} {{{\left( {\frac{{\sum\limits_{k=i-1}^{j - 2} {{\mu _k}({S_k})} }}{{{t^{N - p}}}}} \right)}^{\frac{1}{{p - 1}}}}\frac{{dt}}{t}} } \\&\le\sum\limits_{i = 1}^{j - 1} {{{\left( {\sum\limits_{k=i-1}^{j-2} {r_{k - 1}^{N - p}} } \right)}^{\frac{1}{{p - 1}}}}} r_i^{ - \frac{{N - p}}{{p - 1}}} \le c_{7}j, \text{ with } c_{7}={\left( {\frac{{{4^{N - p}}}}{{1 - {2^{ - (N - p)}}}}} \right)^{\frac{1}{{p - 1}}}}.
      \end{align*}
     Therefore, 
     \begin{align}\nonumber
      {A_2} &\le \sum\limits_{j = 1}^\infty  {\int_{{S_j}} {\exp } \left( {c_3 c_7\varepsilon j} \right)dx}  = \sum\limits_{j = 1}^\infty  {r_j^N\exp \left( {c_3c_{7}{\varepsilon }j} \right)|S_1|}\\&\label{3hvI2}=\sum\limits_{j = 1}^\infty  {\exp \left( {\left( {c_3c_{7}\varepsilon  - N\log (2)} \right)j} \right)|S_1|} \leq c_{8}  ~~\textrm{ for}~~  \varepsilon \leq N\log (2)/(2c_3c_7).
     \end{align}    
Consequently, from \eqref{3hvI1}, \eqref{3hvI2} and \eqref{3hvI3}, we obtain $A\le c_2:=c_6+c_{8}+c_5$ for $\varepsilon=\varepsilon(N,p)$ small enough. This implies 
\begin{equation}\label{3hv06052}
       {\left\| {\exp \left( {\frac{p}{2N}\varepsilon{\bf W}_{1,p}^1\left[\sum\limits_{k = 3}^J { {\mu _k}} \right]} \right)} \right\|_{\mathcal{M}^{\frac{2N}{p}}(B_1)}} \le {c_{9}}{\left( {\int_{{B_1}} {\exp \left( {\varepsilon{\bf W}_{1,p}^1\left[\sum\limits_{k = 3}^J { {\mu _k}} \right](x)} \right)}}dx \right)^{\frac{p}{{2N}}}} \le {c_{10}},
       \end{equation}
       where the constant $c_{10}$ does not depend on $J$. 
Set $B=B_{\frac{1}{4}}$. For $\varepsilon_0=(\frac{p\varepsilon}{2NC})^{1/(p-1)}$, where $C$ is the constant in \eqref{3hvine},  by Theorem \ref{3hv06054} and estimate \eqref{3hv06052},  there exists a nonnegative renormalized solution $u$ to equation 
\begin{equation*}
 \begin{array}{ll}
  - \Delta_p u + e^{ u}-1 =\varepsilon_0 \sum_{j=3}^{J} \mu_{j} \qquad&\text{in }\;B, 
  \\ \phantom{- \Delta_p  + e^{u}-1}
 u = 0&\text{in }\;\partial B,  \\ 
 \end{array} 
\end{equation*} 
satisfying \eqref{3hvine} with $\mu=\varepsilon_0 \sum_{j=3}^{J} \mu_{j}$. 
Thus, from Corollary \ref{3hvcoro1} and estimate \eqref{3hv06052}, we have 
\begin{equation*}
u(0)\geq -c_{11}+c_{12}\mathbf{W}^{\frac{1}{4}}_{1,p}\left[\sum_{j=3}^{J} \mu_{j}\right](0).
\end{equation*}
 Therefore
 \begin{align*}
  u(0) &\geq - {c_{11}} + {c_{12}}\sum\limits_{i = 2}^\infty  {\int\limits_{{r_{i + 1}}}^{{r_i}} {{{\left( {\frac{{\sum\limits_{j = 3}^J {{\mu _j}} ({B_t}(0))}}{{{t^{N - p}}}}} \right)}^{\frac{1}{{p - 1}}}}\frac{{dt}}{t}} }\geq  - {c_{11}} + {c_{12}}\sum\limits_{i = 2}^{J - 2} {\int\limits_{{r_{i + 1}}}^{{r_i}} {{{\left( {\frac{{{\mu _{i + 2}}({B_t}(0))}}{{{t^{N - p}}}}} \right)}^{\frac{1}{{p - 1}}}}\frac{{dt}}{t}} }  \\&=  
      - {c_{11}} + {c_{12}}\sum\limits_{i = 2}^{J - 2} {\int\limits_{{r_{i + 1}}}^{{r_i}} {{{\left( {\frac{{{\mu _{i + 2}}({S_{i + 2}})}}{{{t^
    {N - p}}}}} \right)}^{\frac{1}{{p - 1}}}}\frac{{dt}}{t}} }\geq   - {c_{11}} + {c_{13}}\sum\limits_{i = 2}^{J - 2} {{{\left( {{\cal H}_1^{N - p}(K \cap {S_{i + 2}})} \right)}^{\frac{1}{{p - 1}}}}r_i^{ - \frac{{N - p}}{{p - 1}}}}  \\&= 
       - {c_{11}} + {c_{13}}\sum\limits_{i = 4}^\infty  {{{\left( {{\cal H}_1^{N - p}(K \cap {S_i})} \right)}^{\frac{1}{{p - 1}}}}r_i^{ - \frac{{N - p}}{{p - 1}}}}.  
 \end{align*}
 From the inequality 
  $${\left( {{\cal H}_1^{N - p}(K \cap {S_i})} \right)^{\frac{1}{{p - 1}}}} \geq {\textstyle{1 \over {\max (1,{2^{\frac{{2 - p}}{{p - 1}}}})}}}{\left( {{\cal H}_1^{N - p}(K \cap {B_{i - 1}})} \right)^{\frac{1}{{p - 1}}}} - {\left( {{\cal H}_1^{N - p}(K \cap {B_i})} \right)^{\frac{1}{{p - 1}}}}\quad\forall i,$$
 we deduce that
 \begin{align*}
  u(0) &\geq  - {c_{11}} + {c_{13}}\sum\limits_{i = 4}^\infty  {\left( {{\textstyle{1 \over {\max (1,{2^{\frac{{2 - p}}{{p - 1}}}})}}}{{\left( {{\cal H}_1^{N - p}(K \cap {B_{i - 1}})} \right)}^{\frac{1}{{p - 1}}}} - {{\left( {{\cal H}_1^{N - p}(K \cap {B_i})} \right)}^{\frac{1}{{p - 1}}}}} \right)r_i^{ - \frac{{N - p}}{{p - 1}}}}  
    \\&  \geq  - {c_{11}} + {c_{13}}\left( {{\textstyle{{{2^{\frac{{N - p}}{{p - 1}}}}} \over {\max (1,{2^{\frac{{2 - p}}{{p - 1}}}})}}} - 1} \right)\sum\limits_{i = 4}^\infty  {{{\left( {{\cal H}_1^{N - p}(K \cap {B_i})} \right)}^{\frac{1}{{p - 1}}}}r_i^{ - \frac{{N - p}}{{p - 1}}}} 
    \\&  \geq  - {c_{14}} + {c_{15}}\int\limits_0^1 {{{\left( {\frac{{{\cal H}_1^{N - p}(K \cap {B_t})}}{{{t^{N - p}}}}} \right)}^{\frac{1}{{p - 1}}}}\frac{{dt}}{t}}.
 \end{align*}
Since $U_1$ is the maximal solution in $K^c$, $u$ satisfies the same equation in $B\backslash K$  and $U_1\geq u=0$ on $\partial B$, it follows that $U_1$ dominates $u$ in $B\backslash K$.  Then $U_1(0)\geq u(0) $ and we obtain \eqref{3hvinea}.\smallskip

\noindent{\it 2}. By \cite[Theorem 2.5.3]{33AH}, there exists $\mu_j\in \mathfrak{M}^+(\mathbb{R}^N)$
such that $\text{supp} (\mu_j) \subset K\cap S_j$ and  \[{\mu _j}(K \cap {S_j}) = \int\limits_{\mathbb{R}^N} {{{\left( {{G_p}[{\mu _j}](x)} \right)}^{\frac{q_1}{p-1}}}}dx  = C_{{p,\frac{{{q_1}}}{{{q_1} - p + 1}}}}(K \cap {S_j}).\]
By Jensen's inequality, we have for any $a_k\geq 0$,  
 \[{\left( {\sum\limits_{k = 0}^\infty  {{a_k}} } \right)^s} \le \sum\limits_{k = 0}^\infty  {{\theta _{k,s}}a_k^s} \]
 where $\theta _{k,r}$ has the following expression with $\theta>0$,  
 \[{\theta _{k,s}} = \left\{ \begin{array}{lll}
  1\quad &\text {if} \;s \in (0,1], \\ 
  {\left( {\frac{{\theta  + 1}}{\theta }} \right)^{s - 1}}{\left( {\theta  + 1} \right)^{k(s-1)}}\quad &\text {if} \;s  > 1. \\ 
  \end{array} \right.\]
  Thus, 
  \begin{align*}
   \int\limits_{{B_1}} {{{\left( {{\bf W}_{1,p}^1\left[\sum\limits_{k = 3}^J {{\mu _k}} \right](x)} \right)}^{{q_1}}}}dx&\leq \int_{B_1}\left( \sum_{k=3}^{J}\theta_{k,\frac{1}{p-1}}{\bf W}_{1,p}^1[{\mu _k}](x)\right)^{q_1}dx   \\&\leq \sum\limits_{k = 3}^J {\theta _{k,\frac{1}{{p - 1}}}^{{q_1}}{\theta _{k,{q_1}}}\int\limits_{{B_1}} {{{\left( {{\bf W}_{1,p}^1[{\mu _k}](x)} \right)}^{{q_1}}}} dx }  
        \\&
      \leq {c_{16}}\sum\limits_{k = 3}^J {\theta _{k,\frac{1}{{p - 1}}}^{{q_1}}{\theta _{k,{q_1}}}\int\limits_{\mathbb{R}^N} {{{\left( {{G_p}*{\mu _k}(x)} \right)}^{{\frac{q_1}{p-1}}}}} dx }     \\&
      = {c_{16}}\sum\limits_{k = 3}^J {\theta _{k,\frac{1}{{p - 1}}}^{{q_1}}{\theta _{k,{q_1}}}C_{p,\frac{{{q_1}}}{{{q_1} - p + 1}}}(K \cap {S_k})}     \\ &
      \le {c_{17}}\sum\limits_{k = 3}^J {\theta _{k,\frac{1}{{p - 1}}}^{{q_1}}{\theta _{k,{q_1}}}{2^{ - k\left( {N - \frac{{p{q_1}}}{{{q_1} - p + 1}}} \right)}}}      \\&
      \le {c_{18}}, 
  \end{align*} 
for $\theta$ small enough. Here the third inequality follows from \cite[Theorem 2.3]{33BVHV} and the constant $c_{18}$ does not depend on $J$. 
   Hence, \begin{equation}
  \label{3hv06053}{\left\| {\left( {{\bf W}_{1,p}^1\left[\sum\limits_{k = 3}^J {{\mu _k}} \right]} \right)^q} \right\|_{{{\cal M}^{\frac{{{q_1}}}{q}}}({B_1})}} \le {c_{19}}{\left\| {{\bf W}_{1,p}^1\left[\sum\limits_{k = 3}^J {{\mu _k}}\right]} \right\|_{{L^{{q_1}}}({B_1})}^q} \le {c_{20}},
   \end{equation}
   where $c_{20}$ is independent of $J$.
Take $B=B_{\frac{1}{4}}$. Since $\sum_{j=3}^{J} \mu_{j}$ is absolutely continuous with respect to the capacity $C_{p,\frac{q}{q+1-p}}$ in $B$, thus by Theorem \ref{3hv06054}, there exists a nonnegative renormalized solution $u$ to equation 
\begin{equation*}
\begin{array}{lll}
  - \Delta_p u + u^q= \sum_{j=3}^{J} \mu_{j} ~~&\text{ in }\,B, \\
  \phantom{  - \Delta_p  + u^q} 
 u = 0 &\text{ on }\,\partial B. 
 \end{array} 
\end{equation*} 
satisfying \eqref{3hvine} with $\mu= \sum_{j=3}^{J} \mu_{j}$. 
Thus, from Corollary \ref{3hvcoro1} and estimate \eqref{3hv06053}, we have
  $$u(0)\geq -c_{21}+c_{22}\mathbf{W}^{\frac{1}{4}}_{1,p}\left[\sum_{j=3}^{J} \mu_{j}\right](0).$$
As above, we also get that 
\begin{equation*}
u(0)\geq -c_{23}+c_{24}\int_0^1 {{{\left( {\frac{{C_{{p,\frac{{{q_1}}}{{{q_1} - p + 1}}}}(K \cap B_r)}}{{{r^{N - p}}}}} \right)}^{\frac{1}{{p - 1}}}}} \frac{{dr}}{r}.
\end{equation*}
After we also have $U_2(0)\geq u(0)$. Therefore, we obtain\eqref{3hvineb}.
\end{proof}


\section{Proof of the main results}
First, we prove theorem \textbf{1} in the case case $p=N$. To do this we consider the function
\begin{align*}
x\mapsto U(x)=U(|x|)=\log\left(\frac{N-1}{2^{N+1}}\frac{1}{R^N}\left(\frac{R}{|x|}+1\right)\right)~~\text{ in }~ B_R(0)\backslash\{0\}.
\end{align*} One has 
\begin{align*}
U^{'}(|x|)=\frac{1}{R+|x|}-\frac{1}{|x|}~~\text{ and }~U^{''}(|x|)=-\frac{1}{(R+|x|)^2}+\frac{1}{|x|^2},
\end{align*}
thus, for any $0<|x|<R$,
\begin{align*}
-\Delta_NU +e^U-1&=-(N-1)|U^{'}(|x|)|^{N-2}\left(U^{''}(|x|)+\frac{1}{|x|}U^{'}(|x|)\right) +e^U-1\\&=-\frac{(N-1)R^{N-1}}{(R+|x|)^N|x|^{N-1}}+\frac{N-1}{2^{N+1}}\frac{1}{R^N}\left(\frac{R}{|x|}+1\right)-1
\\& \leq -\frac{(N-1)R^{N-1}}{(2R)^N|x|^{N-1}}+\frac{N-1}{2^{N+1}}\frac{1}{R^N}\frac{2R}{|x|}
\\& \leq -1.
\end{align*}
Hence, if $u\in C^1(\Omega)$ is the maximal solution of
 \begin{equation*}
  -\Delta_N u +e^{u}-1=0~~\text{in }\Omega
  \end{equation*}
  and $R=2\text{diam}(\Omega)$,
 then 
  $u(x)\geq U(|x-y|)$ for any $x\in\Omega$ and $y\in\partial\Omega$. Therefore, $u$ is a large solution and satisfies 
  \begin{align*}
  u(x)\geq \log\left(\frac{N-1}{2^{N+1}}\frac{1}{R^N}\left(\frac{R}{\rho(x)}+1\right)\right)~~\forall ~x\in\Omega.
  \end{align*}
  Now, we prove Theorem \textbf{1} in the case $p<N$ and Theorem \textbf{2}.
 Let $u,v\in C^1(\Omega)$ be the maximal solutions of
 $$\begin{array} {lll}
&(i)\qquad -\Delta_p u +e^{u}-1=0\qquad&\text{in }\Omega,\\[2mm]
\phantom{-1}
 &(ii)\qquad  -\Delta_p v +v^q=0\qquad&\text{in }\Omega.
\end{array}$$
Fix $x_0\in \partial \Omega $. We can assume that $x_0=0$. Let $\delta\in (0,1/12)$. For $z_0\in \overline{B}_\delta\cap \Omega $. Set $K=\Omega^c\cap\overline{B_{1/4}(z_0)}$. Let $U_1,U_2\in C^{1}(K^c)$  be the maximal solutions of \eqref{3hvb-1} and \eqref{3hvb-2} respectively.  We have $u\geq U_1$ and $v \geq U_2$ in $\Omega$. By Theorem \ref{3hvth3}, 
$$\begin{array} {ll}\displaystyle
U_1(z_0)\geq -c_1+c_2\int_{\delta}^{1}\left(\frac{\mathcal{H}_1^{N-p}(K\cap B_r(z_0))}{r^{N-p}}\right)^{\frac{1}{p-1}}\frac{dr}{r}
\\\displaystyle\phantom{U(z_0)}
\geq   -c_1+c_2\int_{\delta}^{1}\left(\frac{\mathcal{H}_1^{N-p}(K\cap B_{r-|z_0|})}{r^{N-p}}\right)^{\frac{1}{p-1}}\frac{dr}{r} \quad(\textrm{since } B_{r-|z_0|}\subset B_r(z_0)))
\\\displaystyle\phantom{U(z_0)}
\geq   -c_1+c_2\int_{2\delta}^{1}\left(\frac{\mathcal{H}_1^{N-p}(K\cap B_{\frac{r}{2}})}{r^{N-p}}\right)^{\frac{1}{p-1}}\frac{dr}{r}
\\\displaystyle\phantom{U(z_0)}
\geq   -c_1+c_3\int_{\delta}^{1/2}\left(\frac{\mathcal{H}_1^{N-p}(K\cap B_r)}{r^{N-p}}\right)^{\frac{1}{p-1}}\frac{dr}{r}.
\end{array}$$
We deduce 
$$\displaystyle\inf_{B_\delta\cap \Omega}u\geq\inf_{B_\delta\cap \Omega}U_1\geq -c_1+c_3\int_{\delta}^{1/2}\left(\frac{\mathcal{H}_1^{N-p}(K\cap B_r)}{r^{N-p}}\right)^{\frac{1}{p-1}}\frac{dr}{r}\to \infty\quad\text{as }\delta \to 0.$$ 
Similarly, we also obtain 
$$\displaystyle\inf_{B_\delta\cap \Omega}v\geq -c_4+c_5\int_{\delta}^{1/2}\left(\frac{C_{p,\frac{q_1}{q_1-p+1}}(K\cap B_r)}{r^{N-2}}\right)^{\frac{1}{p-1}}\frac{dr}{r}\to \infty\quad\text{as }\delta \to 0.$$ 
Therefore, $u$ and $v$ satisfy \eqref{3hvIn9} and  \eqref{3hvIn8} respectively. This completes the proof. \medskip

\section{Large solutions of quasilinear Hamilton-Jacobi equations}

Let $\Omega$ be a bounded open subset of $\mathbb{R}^N$ with $N\geq 2$. In this section we use our previous results to give sufficient conditions for existence of solutions to the problem
  \begin{equation} \label{3hvHJ}\begin{array}{ll}
  -\Delta_pu+a\abs{\nabla u}^q+bu^{s}=0~~&\text{in }\Omega,\\
 \! \displaystyle\phantom{-,,,---}\lim_{\rho(x)\to 0}u(x)=\infty,
\end{array} \end{equation}
where $a>0,b> 0$ and  $1\leq q<p\leq 2$, $q>p-1, s\geq p-1$. 
\smallskip

First we have the result of existence solutions to equation \eqref{3hvHJ}. 
  
  \begin{proposition}
   \label{3hvmaxso} Let $a>0,b>0$ and $q>p-1, s\geq p-1$, $1\le q\le p$ and $1<p\le 2$. There exists a maximal nonnegative solution $u\in {C^1}(\Omega ) $ to equation
    \begin{eqnarray}\label{3hvHJ+}
    &&-\Delta_p u+a\abs{\nabla u}^q+bu^s=0~~\text{in }\Omega,
    \label{3hveqma}
    \end{eqnarray} 
     which satisfies
    \begin{equation}\label{3hv4}
    u(x)\leq c(N,p,s)b^{-\frac{1}{s-p+1}}d(x,\partial \Omega)^{-\frac{p}{s-p+1}}~~ \forall x\in \Omega,
    \end{equation}
    if $s>p-1$,

    \begin{equation}\label{3hv5}
    u(x)\leq c(N,p,q)\left(a^{-\frac{1}{q-p+1}}d(x,\partial \Omega)^{-\frac{p-q}{q-p+1}}+a^{-\frac{1}{q-p+1}}b^{-\frac{1}{p-1}}d(x,\partial \Omega)^{-\frac{q}{(p-1)(q-p+1)}}\right) ~~\forall x\in \Omega,
    \end{equation}
    if  $p-1<q<p$ and $s=p-1$,
    and 
    \begin{equation}\label{3hv6}
    u(x)\leq c(N,p)a^{-1}b^{-\frac{1}{p-1}}d(x,\partial \Omega)^{-\frac{p}{p-1}}~~ \forall x\in \Omega,
    \end{equation}
    if $q=p$ and $s=p-1$.
  \end{proposition} 
  \begin{proof} Case $s=p-1$ and $p-1<q<p$. We consider 
  $$U_1(x)=U_1(|x|)=c_1\left(\frac{R^{p'}-|x|^{p'}}{p' R^{p'-1}}\right)^{-\frac{p-q}{q-p+1}}+c_2\in C^1(B_R(0)).$$
  with $p'=\frac{p}{p-1}$ and $c_1,c_2>0$. 
  We have 
  \begin{align*}
 & U_1^{'}(|x|)=\frac{c_1(p-q)}{q-p+1}\frac{|x|^{p'-1}}{R^{p'-1}}\left(\frac{R^{p'}-|x|^{p'}}{p'R^{p'-1}}\right)^{-\frac{1}{q-p+1}},\\& 
U_1^{''}(|x|)=\frac{c_1(p-q)(p'-1)}{q-p+1}\frac{|x|^{p'-2}}{R^{p'-1}}\left(\frac{R^{p'}-|x|^{p'}}{p'R^{p'-1}}\right)^{-\frac{1}{q-p+1}}\\&~~~~~~~~~~~~~~~~~~+\frac{c_1(p-q)}{(q-p+1)^2}\left(\frac{|x|^{p'-1}}{R^{p'-1}}\right)^2\left(\frac{R^{p'}-|x|^{p'}}{p'R^{p'-1}}\right)^{-\frac{1}{q-p+1}-1}
  \end{align*} and 
  \begin{align*}
  A=-\Delta_p U_1+a|\nabla U_1|^q+bU_1^{p-1}\geq  -\Delta_p U_1+a|\nabla U_1|^q+bc_2^{p-1}.
  \end{align*}
  Thus,  for all $x\in B_R(0)$ 
  \begin{align*}
  A& \geq -(p-1)|U^{'}_1(|x|)|^{p-2}U_1^{''}(|x|)-\frac{N-1}{|x|}|U^{'}_1(|x|)|^{p-2}U^{'}_1(|x|)+a|U^{'}_1(|x|)|^{q}+bc_1^{p-1} \\&
  = \left(\frac{c_1(p-q)(p'-1)}{q-p+1}\right)^{p-1} \left(\frac{R^{p'}-|x|^{p'}}{p'R^{p'-1}}\right)^{-\frac{q}{q-p+1}}\left\{-(p-1)\frac{p'-1}{p'}\left(1-\left(\frac{|x|}{R}\right)^{p'}\right)\right.\\&~~~~~~~~~~~~~
 -\frac{1}{q-p+1}\left(\frac{|x|}{R}\right)^{p'}-\frac{N-1}{p'}\left(\frac{|x|}{R}\right)^{p'}\left(1-\left(\frac{|x|}{R}\right)^{p'}\right)\\&~~~~~~~~~~~~~\left.+a\left(\frac{c_1(p-q)}{q-p+1}\right)^{q-p+1}\left(\frac{|x|}{R}\right)^{\frac{q}{q-p+1}}\right\}+bc_2^{p-1}
 \\&\geq  \left(\frac{c_1(p-q)(p'-1)}{q-p+1}\right)^{p-1} \left(\frac{R^{p'}-|x|^{p'}}{p'R^{p'-1}}\right)^{-\frac{q}{q-p+1}}\\&~~~~\times\left\{-\frac{N(p-1)}{p}-\frac{1}{q-p+1}+ a\left(\frac{c_1(p-q)}{q-p+1}\right)^{q-p+1}\left(\frac{|x|}{R}\right)^{\frac{q}{q-p+1}}\right\}+bc_2^{p-1}.
  \end{align*}
  Clearly, one can find $c_1=c_2(N,p,q)a^{-\frac{1}{q-p+1}}>0$ and $c_3=c_3(N,p,q)>0$ such that 
  \begin{align*}
  A\geq -c_3 a^{-\frac{p-1}{q-p+1}}R^{-\frac{q}{q-p+1}}+bc_2^{p-1}.
  \end{align*}
Choosing $c_2=c_3^{\frac{1}{p-1}}a^{-\frac{1}{q-p+1}}b^{-\frac{1}{p-1}}R^{-\frac{q}{(p-1)(q-p+1)}}$, we get 
    \begin{equation}\label{3hv1}
    -\Delta_p U_1+a|\nabla U_1|^q+bU_1^{p-1} \geq0 ~\text{ in } B_R(0).
    \end{equation}
Likewise, we can verify that the function $U_2$ below 
 $$U_2(x)=c_4a^{-1}\log\left(\frac{ R^{p'}}{R^{p'}-|x|^{p'}}\right)+c_4a^{-1}b^{-\frac{1}{p-1}}R^{-\frac{p}{p-1}}$$
belongs to $C^1_+(B_R(0))$ and satisfies
     \begin{equation}
     \label{3hv2}-\Delta_p U_2+a|\nabla U_2|^p +bU_2^{p-1}\geq0 ~\text{ in } B_R(0).
     \end{equation}
While, if   $s>p-1$,
      $$U_3(x)=c_5b^{-\frac{1}{s-p+1}}\left(\frac{R^{\beta}-|x|^{\beta}}{\beta R^{\beta-1}}\right)^{-\frac{p}{s-p+1}}$$ 
   belongs to  $C^1(B_R(0))$ and verifies
      \begin{equation}
      \label{3hv3}-\Delta_p U_3+bU_3^s \geq0 ~\text{ in } B_R(0),
      \end{equation}
       for some  positive constants $c_4=c_4(N,p,q)$,  $c_5=c_5(N,p,s)$ and $\beta=\beta(N,p,q)>1$.\\We emphasize the fact that with the condition $1<p\le  2$ and $q\geq 1$,  equation \eqref{3hveqma} satisfies a comparison principle, see \cite[Theorem 3.5.1, corollary 3.5.2]{33PUSE}.   
  Take a sequence of smooth domains $\Omega_n$ satisfying  $\Omega_n\subset \overline{\Omega}_n\subset\Omega_{n+1}$ for all $n$ and $\bigcup\limits_{n = 1}^\infty  {{\Omega _n}}  = \Omega$. For each $n,k\in \mathbb{N}^*$, there exist nonnegative solution $u_{n,k}=u\in W^{1,p}_k(\Omega_n):= W^{1,p}_0(\Omega_n)+k$ of equation \eqref{3hveqma} in $\Omega_n$.\\
  Since $-\Delta_p u_{k,n}\le 0$ in $\Omega_n$, so using the maximum principle we get  $u_{n,k}\le k$ in $\Omega_n$ for all $n$. Thus, by standard regularity (see \cite{33Bi} and \cite{33Li}),  $u_{n,k}\in C^{1,\alpha}(\overline{\Omega_n})$ for some $\alpha\in (0,1)$. It follows from the comparison principle and \eqref{3hv1}-\eqref{3hv3},  that 
  
  $$u_{n,k}\le u_{n,k+1}\qquad\text{in }\;\Omega_n$$ 
  and \eqref{3hv4}-\eqref{3hv6} are satisfied with $u_{n,k}$ and $\Omega_n$ in place of  $u$ and $\Omega$ respectively. 
  From this, we derive uniform local bounds for $\{u_{n,k}\}_k$, and by standard interior regularity (see \cite{33Bi}) we obtain uniform local bounds for $\{u_{n,k}\}_k$ in $C^{1,\eta}_{loc}(\Omega_n)$. It implies that the sequence $\{u_{n,k}\}_k$ is pre-compact in $C^{1}$. Therefore, up to a subsequence, $u_{n,k}\to u_n$ in $C^{1}(\Omega_n)$. Hence, we can verify that $u_n$ is a solution of \eqref{3hveqma} and satisfies \eqref{3hv4}-\eqref{3hv6}  with  $u_n$ and $\Omega_n$ replacing $u$ and $\Omega$ and $u_n(x)\to \infty$ as $d(x,\Omega_n)\to 0$.\\
  Next, since $u_{n,k}\geq u_{n+1,k}$ in $\Omega_n$ there holds $u_n\geq u_{n+1}$ in $\Omega_n$. In particular, $\{u_n\}$ is  uniformly locally bounded  in $\Omega$. Arguing as  above, we obtain $u_n\to u$ in $C^1(\Omega)$, thus $u$ is a solution of \eqref{3hveqma} in $\Omega$ and satisfies \eqref{3hv4}-\eqref{3hv6}. Clearly, $u$ is the maximal solution of \eqref{3hveqma}. 
  \end{proof}
\begin{theorem}\label{3hvthHJ2} Let $q_1>p-1$ and $1<p\leq 2$.  Assume that equation \eqref{3hvIn8} admits a solution with $q=q_1$.  Then for any $a>0, b> 0$ and $q\in(p-1,\frac{pq_1}{q_1+1})$,  $ s\in [p-1,q_1)$ equation \eqref{3hveqma} has a large solution satisfying \eqref{3hv4} and \eqref{3hv5}.
\end{theorem}
\begin{proof}   Assume that equation \eqref{3hvIn8} admits a solution $v$ with $q=q_1$ and set $v=\beta w^\sigma$ with $\beta>0,\sigma\in (0,1)$, then $w>0$ and 
\begin{align*}
-\Delta_pw+(-\sigma+1)(p-1)\myfrac{\abs{\nabla w}^p}{w}+\beta^{q_1-p+1}\sigma^{-p+1}w^{\sigma(q_1-p+1)+p-1}=0~\text{in }\Omega.
\end{align*}
If we impose $\max\{\frac{s-p+1}{q_1-p+1},\left(\frac{q}{p-q}-p+1\right)\frac{1}{q_1-p+1}\}<\sigma<1$, we can see that
\begin{equation*}
(-\sigma+1)(p-1)\myfrac{\abs{\nabla w}^p}{w}+\beta^{q_1-p+1}\sigma^{-p+1}w^{\sigma(q_1-p+1)+p-1}\geq a|\nabla w|^q+bw^s~~\text{in }~\{x:w(x)\geq M\},
\end{equation*}
where a positive constant $M$ depends on $p,q_1,q,s,a,b$. 
Therefore
\begin{equation*}
-\Delta_pw+ a\abs{\nabla w}^q+bw^s\le 0~~~\text{in } ~\{x:w(x)\geq M\}.
\end{equation*}
 Now we take an open subset $\Omega'$ of $\Omega$ with $\overline{\Omega'}\subset\Omega$ such that  the set $\{x:w(x)\geq M\}$ contains  $\Omega\backslash\overline{\Omega'}$. So $w$ is a subsolution of $-\Delta_pu+ a\abs{\nabla u}^q+bu^s = 0$ in $\Omega\backslash\overline{\Omega'}$ and the same property holds with $w_\varepsilon:=\varepsilon w$ for any $\varepsilon\in (0,1)$. Let $u$ be as in Proposition \ref{3hvmaxso}. Set $\min\{u(x):x\in \partial \Omega'\}=\theta_1>0$ and $\max\{w(x):x\in \partial \Omega'\}=\theta_2\geq M$. Thus $w_\varepsilon <u $ on $\partial \Omega'$ with $\varepsilon<\min\{\frac{\theta_1}{\theta_2},1\}$. Hence, from the construction of $u$ in the proof of Proposition \ref{3hvmaxso} and the comparison principle,  we obtain $w_\varepsilon \le u$ in $\Omega\backslash\overline{\Omega'}$. This implies the result. 
\end{proof}

\begin{remark}From the proof of above Theorem, we can show that under the assumption as in Proposition \ref{3hvmaxso}, equation \eqref{3hveqma} has a large solution in $\Omega$ if and only if equation \eqref{3hveqma} has a large solution in $\Omega\backslash K$ for some a compact set $K\subset \Omega$ with smooth boundary.
\end{remark}

Now we deal with  \eqref{3hvHJ} in the case $q=p$.
\begin{theorem} Assume that equation \eqref{3hvHJ+} has a large solution in $\Omega$ for some $a,b> 0$, $s> p-1$ and $q=p>1$.
Then for any $a_1,b_1>0$ and $q_1>p-1,s_1\geq p-1$, $1\le q_1\le p\leq 2$, equation \eqref{3hveqma} also has a large solution $u$ in $\Omega$ with parameters $a_1,b_1,q_1,s_1$ in place of  $a,b,q,s$ respectively, and it satisfies \eqref{3hv4}-\eqref{3hv6}.
\end{theorem}
\begin{proof}
For $\sigma>0$ we set $u= v^\sigma$
 thus
 \begin{equation*}
 -\Delta_p v-(\sigma-1)(p-1)\frac{\abs{\nabla v}^p}{v}+a \sigma v^{\sigma-1}\abs{\nabla v}^p+b\sigma^{-p+1}v^{(s-p+1)\sigma+p-1}=0.
 \end{equation*}
 Choose $\sigma=\frac{s_1-p+1}{s-p+1}+2$, it is easy to see that
 \begin{equation*}
 -\Delta_p v+a_1|\nabla v|^{q_1}+b_2v^{s_1}\le 0 ~\text{ in }~ \{x:v(x)\geq M\},
 \end{equation*}
  for some a positive constant $M$ only depending on $p,s, a,b,a_1,b_1,q_1,s_1$.
Similarly as in the proof of Theorem \ref{3hvthHJ2}, we get the result as desired. 
\end{proof}
\begin{remark}If we set $u=e^v$ then $v$ satisfies
    \begin{equation*}\begin{array}{ll}
-\Delta_p v+be^{(s-p+1)v}=\abs{\nabla v}^p(p-1-ae^v)\qquad\text{in }\Omega.
\end{array} \end{equation*}
From this, we can construct a large solution of 
    \begin{equation*}\begin{array}{ll}
-\Delta_p u+be^{(s-p+1)u}=0\qquad\text{in }\Omega\backslash K,
\end{array} \end{equation*}
for any a compact set $K\subset \Omega$ with smooth boundary such that $v\geq ln \left(\frac{p-1}{a}\right)$ in  $\Omega\backslash K$. In case $p=2$,
It would be interesting to see what Wiener type criterion is implied by the existence as such a large solution. We conjecture that this condition must be
\begin{equation*}
\int_{0}^{1}\myfrac{\mathcal{H}_1^{N-2}(B_r(x)\cap \Omega^c)}{r^{N-2}}\frac{dr}{r}=\infty\qquad\forall x\in \partial \Omega.
\end{equation*}
\end{remark}
We now consider the function 
\begin{equation*}
U_4(x)=c\left(\frac{R^\beta-|x|^\beta}{\beta R^{\beta-1}}\right)^{\frac{p}{\gamma+p-1}}~\text{ in } B_R(0), \gamma>0.
\end{equation*}
As in the proof of proposition \ref{3hvmaxso}, it is easy to check  that there exist positive constants $\beta$ large enough and $c$ small enough so that inequality
 $\Delta_p U_4+U^{-\gamma}_4\geq 0$ holds.\\
From this, we get the existence of minimal solution to equation
\begin{equation}\label{3hvquench1}\begin{array} {ll}
\Delta_p u+u^{-\gamma}=0\qquad\text{in}\;\Omega.
\end{array}\end{equation} 
\begin{proposition} Assume $\gamma>0$. Then there exists a minimal solution $u\in C^1(\Omega)$ to equation \eqref{3hvquench1} and it satisfies  
     $u(x)\geq C d(x,\partial\Omega)^{\frac{p}{\gamma+p-1}}$ in $\Omega$. 
\end{proposition}

We can verify that if the boundary of $\Omega$ is satisfied \eqref{3hvIn3}, then above minimal solution $u$ belongs to 
$C(\overline{\Omega})$, vanishes on $\partial\Omega$ and it is therefore a solution to the quenching problem
\begin{equation}\label{3hvquench2}\begin{array} {ll}
\Delta_p u+u^{-\gamma}=0\qquad\text{in}\;\Omega,\\
\phantom{\Delta_p +u^{-\gamma}}
u=0\qquad\text{in}\;\partial\Omega.
\end{array}\end{equation} 
 
\begin{theorem} \label{3hvend} Let $\gamma>0$. 
Assume that there exists a solution $u\in C(\overline{\Omega})$ to  problem \eqref{3hvquench2}.
Then, for any $a,b>0$ and $q>p-1, s\geq p-1$, $1\le q\le p\leq 2$, equation \eqref{3hveqma}  admits a large solution in $\Omega$ and it satisfies \eqref{3hv4}-\eqref{3hv6}. 
\end{theorem} 
\begin{proof}
We set $u=e^{-\frac{a}{p-1}v}$, then $v$ is a large solution of 
    \begin{equation*} \begin{array}{ll}
-\Delta_p v+a\abs{\nabla v}^p+\left(\frac{p-1}{a}\right)^{p-1}e^{\frac{a}{p-1}(\gamma+p-1)v}=0\qquad&\text{in }\Omega.
\end{array} \end{equation*}
So  \begin{equation*}\begin{array}{ll}
-\Delta_p v+a\abs{\nabla v}^q+bv^s\le 0\qquad&\text{in } \{x:v(x)\geq M\},
\end{array} \end{equation*}
for some a positive constant $M$ only depending on $p,q,s,a,b,\gamma$. Similarly to the proof of Theorem \ref{3hvthHJ2}, we get the result as desired.
\end{proof}


\end{document}